\documentclass[a4paper]{amsart}
\usepackage[leqno]{amsmath}
\usepackage{amsthm,amssymb,amsfonts}
\usepackage{amscd}
\usepackage{mathtools}
\usepackage{bbm}
\usepackage{color}
\usepackage{enumerate}
\usepackage{cite}
\usepackage[utf8]{inputenc}
\usepackage[all,cmtip]{xy}
\usepackage{etoolbox}
\usepackage{tikz}
\usetikzlibrary{arrows}
\usepackage{extarrows}

\newcommand{\Z}{\ensuremath{\mathbb{Z}}}
\newcommand{\Q}{\ensuremath{\mathbb{Q}}}

\newcommand{\C}{\ensuremath{\mathbb{C}}}

\DeclareMathOperator{\Hom}{Hom}

\DeclareMathOperator{\Aut}{Aut}
\DeclareMathOperator{\id}{id}

\DeclareMathOperator{\SL}{SL}
\DeclareMathOperator{\GL}{GL}
\DeclareMathOperator{\PGL}{PGL}
\DeclareMathOperator{\Gr}{Gr}
\newcommand{\PP}{\ensuremath{\mathbb{P}}}

\DeclareMathOperator{\Ext}{Ext}
\DeclareMathOperator{\St}{St}

\DeclareMathOperator{\ord}{ord}
\DeclareMathOperator{\univ}{un}
\DeclareMathOperator{\geom}{geo}
\DeclareMathOperator{\integ}{int}

\DeclareMathOperator{\Maps}{\mathcal{F}}
\DeclareMathOperator{\Dist}{Dist}
\DeclareMathOperator{\cind}{c-ind}

\DeclareMathOperator{\HH}{H}
\renewcommand{\det}{\operatorname{det}}
		
\newcommand{\G}{\ensuremath{\mathcal{G}}}

\DeclareMathOperator{\cont}{ct}

\newcommand{\n}{\ensuremath{\mathfrak{n}}}

\newcommand{\p}{\ensuremath{\mathfrak{p}}}

\newcommand{\LI}{\mathcal{L}}

\newcommand{\into}{\hookrightarrow}

\newcommand{\too}{\longrightarrow}								
\newcommand{\mapstoo}{\longmapsto}
\newcommand{\intoo}{\lhook\joinrel\longrightarrow}

\DeclareRobustCommand\ontoo{\relbar\joinrel\twoheadrightarrow}

\newtheorem{Lem}{Lemma}
\newtheorem{Pro}[Lem]{Proposition}
\newtheorem{Thm}[Lem]{Theorem}
\newtheorem{Def}[Lem]{Definition}
\newtheorem{Rem}[Lem]{Remark}

\newtheorem{Cor}[Lem]{Corollary}

\author[L.~Gehrmann]{Lennart Gehrmann}
\address{L.~Gehrmann \\ Fakult\"at f\"ur Mathematik \\ Universit\"at Duisburg-Essen \\ Thea-Leymann-Stra\ss e 9 \\ 45127 Essen \\ Germany}
\email{lennart.gehrmann@uni-due.de}

\title[Invertible analytic functions and universal extensions]{Invertible analytic functions on Drinfeld symmetric spaces and universal extensions of Steinberg representations}

\begin{document}
\begin{abstract}
Recently, Gekeler proved that the group of invertible analytic functions modulo constant functions on Drinfeld's upper half space is isomorphic to the dual of an integral generalized Steinberg representation.
In this note we show that the group of invertible functions is the dual of a universal extension of that Steinberg representation.
As an application, we show that lifting obstructions of rigid analytic theta cocycles of Hilbert modular forms in the sense of Darmon--Vonk can be computed in terms of $\LI$-invariants of the associated Galois representation. 
The same argument applies to theta cocycles for definite unitary groups.
\end{abstract}
\maketitle

\tableofcontents


\subsection*{Introduction}
Let $F$ be a local field and $\C_F$ the completion of an algebraic closure of $F$.
In \cite{Drinfeld} Drinfeld introduced the $p$-adic period space $\Omega$ of dimension $n-1$:
it is a rigid analytic variety over $F$ whose set of $\C_F$-points is given by the complement of all $F$-rational hyperplanes in $\PP^{n-1}(\C_F)$.
It carries a natural $\PGL_n(F)$-action.

Let $\mathcal{A}_{\C_F}$ be the ring of analytic functions on $\Omega$ that are defined over $\C_F$.
In \cite{vdP} (see also \cite{FvdP}, Theorem 2.7.11) van der Put constructed for $n=2$ a $\PGL_2(F)$-equivariant isomorphism
$$P\colon \mathcal{A}_{\C_F}^\times/\C_F^{\times} \xlongrightarrow{\cong} \Dist_0(\PP^{1}(F),\Z),$$
where $\Dist_0(\PP^{1}(F),\Z)$ denotes the space of $\Z$-valued distributions on $\PP^{1}(F)$ with total mass $0$.
Recently, Gekeler (see \cite{Gekeler}) generalized van der Put's result to arbitrary dimension and constructed a $\PGL_n(F)$-equivariant isomorphism
$$P\colon \mathcal{A}_{\C_F}^{\times}/\C_F^{\times}  \xlongrightarrow{\cong} \Dist_0(\Gr_{n-1,n}(F),\Z),$$
where $\Gr_{n-1,n}$ denotes the Grassmannian of hyperplanes in $F^{n}$ or, by duality, the projective space of the dual of $F^{n}$.
A similar result holds if one replaces $\mathcal{A}_{\C_F}$ by the ring of analytic functions defined over some extension $E$ of $F$ (see Theorem \ref{Gekeler} for a precise formulation.) 
Even more recently, Junger gave a completely different proof of this result (cf.~\cite{Junger}, Theorem B).

By definition, the space of $\Z$-valued distributions of total mass $0$ on the Grassmannian $\Gr_{n-1,n}(F)$ is the $\Z$-dual of the integral generalized Steinberg representation attached to a maximal parabolic subgroup of $\GL_n$ of type $(n-1,1)$.
The main aim of this note is to determine the class of the extension
$$0 \too \C_F^{\times } \too \mathcal{A}_{\C_F}^{\times}\too \mathcal{A}_{\C_F}^{\times}/\C_F^{\times} \too 0$$
in representation-theoretic terms.
We show that this extension is as non-split as possible.

More precisely, in Section \ref{1} we study extensions of generalized Steinberg representations attached to maximal parabolic subgroups of arbitrary type.
Let $r$ be an integer with $0 < r < n$ and denote by $\Gr_{r,n}(F)$ the Grassmannian of $r$-dimensional subspaces of $F^{n}$.
If $N$ is a topological group we denote by $v_r(N)$ the space of continuous functions from $\Gr_{r,n}(F)$ to $N$ modulo constant functions.
To each continuous homomorphism $\lambda\colon F^{\times}\to N$ we construct an extension
$$0 \too v_r(N) \too \mathcal{E}(\lambda) \too \Z \too 0.$$
These can be viewed as multiplicative refinements of the extensions studied in \cite{Ge4}, Section 2.5, which were inspired by the constructions in Section 2.2 of \cite{Ding}.
Results of Dat (cf.~\cite{Dat}), Orlik (cf.~\cite{Orlik}) and Colmez--Dospinescu--Hauseux--Nizioł (cf.~\cite{CDHN}) imply that the extension attached to the identity $\id\colon F^{\times}\to F^{\times}$ can be regarded as a universal extension (see Proposition \ref{puniv}).

In the second section we relate the $\PGL_n(F)$-module $\mathcal{A}_{\C_F}^{\times}$ to a dual of the universal extension.
Let $\mathcal{E}(\iota)$ be the extension associated with the embedding $\iota\colon F^{\times}\to \C_F^{\times}.$
Taking $\Hom_\Z(\cdot,\C_F^{\times})$ induces an exact sequence
\begin{align}\label{intro1}0 \too \C_F^{\times} \too \Hom_{\Z}(\mathcal{E}(\iota), \C_F^{\times})\too \Hom_{\Z}(v_{n-1}(\C_F^{\times}), \C_F^{\times})\too 0.\end{align}
Basic integration theory (see Section \ref{Integration}) provides a map
\begin{align}\label{basic}
\Hom_\Z(v_{n-1}(\Z),\Z)\too\Hom_{\Z}(v_{n-1}(\C_F^{\times}), \C_F^{\times}).
\end{align}
Let $\mathcal{E}(\iota)^{\vee}$ be the pullback of $\Hom_{\Z}(\mathcal{E}(\iota), \C_F^{\times})$ along \eqref{basic}.
By \eqref{intro1}, it sits inside an exact sequence of the form
$$0 \too \C_F^{\times} \too \mathcal{E}(\iota)^{\vee}\too \Hom_{\Z}(v_{n-1}(\Z), \Z)\too 0.$$
The main result of this article is that there exists a $\PGL_n(F)$-equivariant isomorphism
$\mathcal{E}(\iota)^{\vee}\xlongrightarrow{\cong} \mathcal{A}_{\C_F}^{\times}$
that is compatible with the inverse of Gekeler's isomorphism.
Again, one may replace $\mathcal{A}_{\C_F}$ by the ring of analytic functions defined over some extension $E$ of $F$ (see Theorem \ref{mainthm}).
The case $E=F$ involves the universal extension.

This project started as an attempt to understand the relation between lifting obstructions of rigid analytic theta cocycles in the sense of Darmon--Vonk (cf.~\cite{DV}, \cite{DV2}) and automorphic $\LI$-invariants as introduced by Spie\ss~in \cite{Sp}.
As a first application of our main theorem, we show how lifting obstructions of cuspidal theta cocycles for Hilbert modular groups can be computed in terms of $\LI$-invariants of the associated Galois representations.
We also give an alternative construction of the Dedekind--Rademacher cocycle that was first constructed by Darmon--Pozzi--Vonk using Siegel units (cf.~\cite{DPV2}, Theorem A).
We also discuss the example of theta cocycles for definite unitary groups of arbitrary rank and their connection with automorphic respectively Fontaine--Mazur $\LI$-invariants for higher rank groups as introduced in \cite{Ge4}.
In all of the above cases we first compute the lifting obstruction of theta cocycles in terms of automorphic $\LI$-invariants.
Second, we use the equality of automorphic and Fontaine--Mazur $\LI$-invariants as proven in \cite{Sp3} respectively \cite{GeR}.

\subsection*{Acknowledgements}
It is my pleasure to thank all the participants of the ``$p$-adic Kudla seminar'', which was part of the special semester on number theory at the Centre de Recherches Mathématiques, Montreal.
I thank the organizers of the special semester as well as the staff of the CRM for providing a pleasant work environment amidst the difficult year of 2020. 
While working on this manuscript I was visiting McGill University, supported by Deutsche Forschungsgemeinschaft, and I would like to thank these institutions.
Finally, I thank the referees for their detailed comments that helped to improve the paper immensely.

\subsection*{Notation}
The space of continuous maps from a topological space $X$ to a topological space $Y$ is denoted by $C(X,Y).$
We always endow it with the compact-open topology.
If $A$ and $B$ are topological groups, we write $\Hom(A,B)$ for the space of continuous homomorphism from $A$ to $B$.
All rings will be commutative and unital.
If $R$ is a ring, we write $R^{\times}$ for the group of invertible elements of $R$.

\section{Extensions of generalized Steinberg representations}\label{1}
We construct multiplicative refinements of the extensions of generalized Steinberg representations studied in \cite{Ge4}, Section 2.5.
These refinements were previously constructed in the case $n=2$ in Section 6.1 of \cite{BG2}.

Throughout Section \ref{1} and \ref{2} we fix a non-Archimedean local field $F$ of residue characteristic $p$ and an $F$-vector space $V$ of dimension $n\geq 2$.
If $W$ is any finite-dimensional $F$-vector space, we denote by $\GL_W$ (respectively $\PGL_W$) the general (respectively projective) linear group of $W$ viewed as an algebraic group over $F$.
We endow $\GL_W(F)$ and $\PGL_W(F)$ with the natural topology induced by the one on $F$.
These are locally profinite groups.
We often abbreviate $G=\PGL_V(F).$

\subsection{Generalized Steinberg representations}
We fix an integer $r$ with $0<r<n$ and write $\Gr_{r,V}$ for the Grassmannian variety that parametrizes all $r$-dimensional subspaces of $V$.
We endow $\Gr_{r,V}(F)$ with the natural topology inherited from the one on $F$.
It is a compact, totally disconnected space.
Given an abelian topological group $N$ we define the (continuous) generalized Steinberg representations $v_{r,V}(N)$ as the space of continuous functions from 
$\Gr_{r,V}(F)$ to $N$ modulo constant functions, i.e., 
$$v_{r,V}(N)=C(\Gr_{r,V}(F),N)/N.$$
In the following we often abbreviate $v_{r}(N)=v_{r,V}(N)$.
The group $\PGL_V(F)$ acts on $v_{r}(N)$ via $(g.f)(W)=f(g^{-1}.W)$.
Suppose that $N$ is discrete.
Then the natural map
$$C(X,\Z)\otimes_{\Z}N \xrightarrow{\cong} C(X,N)$$
is an isomorphism for every compact, totally disconnected space $X$ and, therefore, the natural map
$$v_{r}(\Z)\otimes_\Z N\xrightarrow{\cong} v_{r}(N)$$
is an isomorphism.

We fix an $F$-rational point $W_0\in \Gr_{r,V}(F)$.
Its stabilizer $P$ in $\GL_V$ is a maximal proper parabolic subgroup.
The map
$$\GL_V\too \Gr_{r,V},\ g\mapstoo g.W_0$$
induces a $\PGL_V(F)$-equivariant isomorphism
\begin{align}\label{projection} \GL_V\hspace{-0.2em}/P\xlongrightarrow{\cong}\Gr_{r,V}.\end{align}
Thus, we get an isomorphism
$$C(\GL_V(F)/P(F),N)/N \xlongrightarrow{\cong} v_r(N)$$
of $G$-modules.

\begin{Lem}\label{invariants}
Let $N$ be an abelian topological group.
Then we have
$$v_r(N)^{\PGL_V(F)}=0.$$
\end{Lem}
\begin{proof}
Let $\overline{\Phi}$ be an element of $v_r(N)^{\PGL_V(F)}$ and $\Phi\in C(\Gr_{r,V}(F),N)$ a representative of $\overline{\Phi}$.
By \eqref{projection} we may view $\Phi$ as a function on $\GL_V(F)$.
Invariance of $\overline{\Phi}$ implies that for all $g\in \GL_V$ there exists a constant $c(g)\in N$ such that
$$\Phi(gg^{\prime}) = c(g) + \Phi(g^{\prime})$$
holds for all  $g^{\prime}\in \GL_V(F)$.
We may assume that $\Phi(1)=0.$
Thus, $c(g)=\Phi(g)$ and, therefore, $\Phi\colon \GL_N(V)\to N$ is a group homomorphism, which is trivial on $P(F)$.
Since $\SL_V(F)$ is the commutator subgroup of $\GL_V$, we see that $\Phi$ factors over the determinant.
Therefore, $\Phi$ is trivial since it is trivial on $P(F)$.
\end{proof}

\subsection{Integration}\label{Integration}
Let $X$ be a compact, totally disconnected space.
We write $\Dist(X,\Z)$ for the space of $\Z$-valued distributions on $X$, i.e.
$$\Dist(X,\Z)=\Hom_{\Z}(C(X,\Z),\Z).$$
For every abelian prodiscrete group $N$ there exists an integration pairing
\begin{align*}
C(X,N)\otimes \Dist(X,\Z)\too N
\end{align*}
constructed as follows:
for any discrete group $B$ the canonical map
$$C(X,\Z)\otimes B \xlongrightarrow{\cong} C(X,B)$$
is an isomorphism and, thus, we have a canonical pairing
$$ C(X,B)\otimes \Dist(X,\Z)\too B$$
that is functorial in $B$.
Choose a basis of neighbourhoods $\left\{U_i\right\}$ of the identity of $N$ consisting of open subgroups with $U_{i+1}\subseteq U_i$.
Taking projective limits we get a pairing
$$ \varprojlim_{i}C(X,N/U_i)\otimes \Dist(X,\Z)\too \varprojlim_{i} N/U_i=N.$$
The canonical map
$$C(X,N)\xlongrightarrow{\cong} \varprojlim_{i}C(X,N/U_i)$$
is an isomorphism and hence, we constructed the desired pairing.

In particular, for every abelian prodiscrete group $N$ the integration pairing induces a map
$$\integ\colon\Hom_{\Z}(v_{r}(\Z),\Z)\too \Hom_{\Z}(v_{r}(N),N).$$

By similar arguments, we can define a canonical map
$$\Hom_{\Z_p}(v_{r}(\Z_p),\Z_p)\too \Hom_{\Z_p}(v_{r}(N),N)$$
for every abelian pro-$p$ group $N$.

\subsection{Preliminaries on continuous extensions}
Let $R$ be a ring and $H$ a topological group.
A topological $R[H]$-module is a topological abelian group $M$ that is also an $R[H]$-module such that the action
$$H\times M \too M$$
is continuous.
A homomorphism between topological $R[H]$-modules is a homomorphism of the underlying $R[H]$-modules that is continuous.
We say that a sequence
$$0\too M_1 \xlongrightarrow{f} M_2 \xlongrightarrow{g} M_3 \too 0$$
of topological $R[H]$-modules is exact if it is exact in the category of $R[H]$-modules and the map $g$ admits a continuous (but not necessarily linear) section $s$ that induces a homeomorphism
\begin{align}\label{homeomorphism}
(f,s)\colon M_1 \times M_3 \xlongrightarrow{\cong} M_2.
\end{align}
We say that $M_2$ is a continuous extension of $M_1$ by $M_3$ in this case.
Two exact sequences
\begin{align*}
0\too M_1 \longrightarrow M_2 \longrightarrow M_3 \too 0 \\
\intertext{and}
0\too M_1 \longrightarrow \tilde{M}_2 \longrightarrow M_3 \too 0
\end{align*}
are said to be equivalent if there exists an $R[H]$-linear map
$$\varphi\colon M_2 \too \tilde{M}_2$$
such that the diagram
\begin{center}
 \begin{tikzpicture}
    \path 	(0.2,0) node[name=A]{$0$}
		(1.5,0) node[name=B]{$M_1$}
		(3,0) node[name=C]{$M_2$}
		(4.5,0) node[name=D]{$M_3$}
		(5.8,0) node[name=E]{$0$}
		(0.2,-1) node[name=F]{$0$}
		(1.5,-1) node[name=G]{$M_1$}
		(3,-1) node[name=H]{$\tilde{M}_2$}
		(4.5,-1) node[name=I]{$M_3$}
		(5.8,-1) node[name=J]{$0$};
    \draw[->] (A) -- (B) ;
		\draw[->] (B) -- (C) ;
		\draw[->] (C) -- (D) ;
		\draw[->] (D) -- (E) ;
		\draw[->] (F) -- (G) ;
		\draw[->] (G) -- (H) ;
		\draw[->] (H) -- (I) ;
		\draw[->] (I) -- (J) ;
    \draw[->] (B) -- (G) node[midway, right]{$=$};
    \draw[->] (C) -- (H) node[midway, right]{$\varphi$};
    \draw[->] (D) -- (I) node[midway, right]{$=$};
  \end{tikzpicture} 
\end{center}
commutes.
One deduces from the existence of sections fulfilling \eqref{homeomorphism} that $\varphi$ is automatically a homeomorphism.

Given two topological $R[H]$-modules $M_1$ and $M_2$ we write $\Ext_{R[H],\cont}^1(M_1,M_2)$ for the set of continuous extensions of $M_1$ by $M_2$ up to equivalence.
The Baer sum defines the structure of an abelian group on $\Ext_{R[H],\cont}^1(M_1,M_2).$
Forgetting the topology induces an inclusion
$$\Ext_{R[H],\cont}^1(M_1,M_2)\intoo \Ext_{R[H]}^1(M_1,M_2).$$
In case $R=\Z$ the group $\Ext_{\Z[H],\cont}^1(M_1,M_2)$ agrees with the group $\Ext_{\mathcal{S}}^1(M_1,M_2)$ defined in Section 2 of \cite{Stasheff}.
In particular, considering $\Z$ as a discrete topological space with trivial $H$-action, we have for every topological $\Z[H]$-module a natural isomorphism
$$\Ext_{\Z[H],\cont}^1(\Z,M)\cong \HH^1_{\cont}(H,M)$$
where $\HH^1_{\cont}(H,M)$ denotes the first continuous cohomology group.

\subsection{Continuous induction}
Given a topological $R[P(F)]$-module $M$ we define its continuous induction as
$$i_P(M)=\left\{\varphi\colon \GL_V(F)\to M\ \mbox{cont.}\mid \varphi(gp)=p.\phi(g)\ \forall p\in P(F), g\in \GL_V(F)\right\}.$$
The group $\GL_V(F)$ acts on $i_P(M)$ via left translation.
The continuous induction $i_P(M)\subseteq C(\GL_V(M), M)$ becomes a topological $R[\GL_V(F)]$-module by endowing it with the subspace topology.

A homomorphism
$$f\colon M_1 \too M_2$$
of topological $R[P(F)]$-modules induces the homomorphism
$$i_P(f)\colon i_P(M_1)\too i_P(M_2),\ \varphi \mapstoo f\circ\varphi$$
of topological $R[\GL_V(F)]$-modules.

Suppose $\GL_V(F)$ acts trivially on $M$.
Then the induction $i_P(M)$ is by definition the space $C(\GL_V(F)/P(F),M)$ which we identify with $C(\Gr_{r,V}(F),M)$ via \eqref{projection}.
This in turn induces a $\GL_V(F)$-equivariant isomorphism $i_P(M)/M\cong v_{r}(M)$.

\begin{Lem}\label{exact}
For every exact sequence
$$0\too M_1 \xlongrightarrow{f}M_2 \xlongrightarrow{g} M_3 \too 0$$
of topological $R[P(F)]$-modules, the induced sequence
$$0\too i_P(M_1) \xlongrightarrow{i_P(f)}i_P(M_2) \xlongrightarrow{i_P(g)} i_P(M_3) \too 0$$
is an exact sequence of topological $R[\GL_V(F)]$-modules.
\end{Lem}
\begin{proof}
The only non-trivial step is to show that $i_P(g)$ is surjective and that it admits a topological section.

By the Bruhat decomposition the map
$$\pi\colon \GL_V\too \GL_V/P$$
is a locally trivial fibration.
Thus, since $\GL_V(F)/P(F)$ is totally disconnected and compact, we can find finitely many closed subsets $U_i\subseteq \GL_V(F)$ such that the maps
$$U_i \times P(F) \too \GL_V(F),\quad (u,p)\mapstoo u\cdot p$$
are homeomorphisms onto their images, their images are disjoint and form an open cover of $\GL_V(F)$.

Let $\varphi\colon \GL_V(F)\to M_3$ be an element of $i_P(M_3)$.
We construct a preimage $\Phi\colon \GL_V(F)\too M_2$ of $\varphi$ as follows:
every $g$ in $G$ can be uniquely written as a product $u\cdot p$ with $u$ in a unique $U_i$ and $p\in P(F)$.
We put $\Phi(g)=p.s(\varphi(u)),$ where $s\colon M_3 \to M_2$ is a fixed section of $g$.
Thus, the homomorphism $i_P(g)$ is surjective.
Moreover, the map sending $\varphi$ to $\Phi$ defines a topological section of $i_P(g)$.
\end{proof}

\subsection{Extensions}
Let $R$ be a topological ring and $N$ a topological $R$-module, i.e., $N$ is an abelian topological group and an $R$-module such that the multiplication map $R\times N\to N$ is continuous.
We consider both $R$ and $N$ as $P(F)$-modules via the trivial action.

To a continuous group homomorphism $\lambda\colon P(F)\to N$ we attach the topological $R[P(F)]$-module $M_\lambda=N \oplus R$ with the $P(F)$-action given by
$$p.(n,r)=(n+r\cdot\lambda(p),r)$$
for $p\in P(F)$, $n\in N$ and $r\in R$.
By definition $M_\lambda$ is a continuous extension of $R$ by $N$.
By Lemma \ref{exact} the sequence
$$0\too i_P(N)\too i_P(M_\lambda) \too i_P(R) \too 0$$
is an exact sequence of topological $R[\GL_V(F)]$-modules.
Let $\widetilde{\mathcal{E}}(\lambda)$ be the pullback of $i_P(M_\tau)$ along $R\into i_P(R)$, i.e., $\widetilde{\mathcal{E}}(\lambda)$ sits inside an exact sequence
$$0\too i_P(N)\too \widetilde{\mathcal{E}}(\lambda)\too R \too 0$$
of topological $R[\GL_V(F)]$-modules.
More concretely we can identify $\widetilde{\mathcal{E}}(\lambda)$ with the set of pairs $(\Phi,r)\in C(\GL_V(F),N)\times R$ such that
$$\Phi(gp)=\Phi(g) + r\cdot \lambda(p)$$
for all $p\in P(F)$ and $g\in \GL_V(F)$.
The group $\GL_V(F)$ acts via left multiplication on the first factor.
The subspace $\widetilde{\mathcal{E}}(\lambda)_0$ of tuples of the form $(\Phi,0)$ with constant $\Phi$ is $\GL_V(F)$-invariant.
By definition the quotient $\mathcal{E}(\lambda)=\widetilde{\mathcal{E}}(\lambda) /\widetilde{\mathcal{E}}(\lambda)_0$ sits inside an exact sequence
$$0\too v_{r}(N)\too \mathcal{E}(\lambda)\too R \too 0$$
of topological $R[\GL_V(F)]$-modules.
It is easy to see that the centre of $\GL_V(F)$ acts trivially on $\mathcal{E}(\lambda)$.
Let $b_{\lambda,R}$ be the associated class in $\Ext^{1}_{R[G],\cont}(R,v_{r}(N))$.
\begin{Lem}
The map $$\Hom(P(F),N)\too \Ext^{1}_{R[G],\cont}(R,v_{r}(N)),\quad \lambda\mapstoo b_{\lambda,R}$$
is a group homomorphism that is functorial in $R$ and $N$.
\end{Lem}
\begin{proof}
Functoriality in $R$ and $N$ follows directly from the construction.
Thus we only have to show that the map is a group homomorphism:
let $\lambda, \lambda^{\prime}\colon P(F)\to A$ be continuous homomorphisms.
The Baer sum of the two extensions $\mathcal{E}(\lambda)$ and $\mathcal{E}(\lambda^{\prime})$ is the space of triples $(\overline{\Phi_1},\overline{\Phi_2},r)$ with  $(\overline{\Phi_1},r)\in \mathcal{E}(\lambda)$ and $(\overline{\Phi_2},r)\in\mathcal{E}(\lambda^{\prime})$) modulo triples of the form $(\overline{\Phi},-\overline{\Phi},0)$ with $\overline{\Phi} \in v_{r}(A)$.
Sending a triple $(\overline{\Phi_1},\overline{\Phi_2},r)$ to the tuple $(\overline{\Phi_1+\Phi_2},r)$ defines a map from the Baer sum to $\mathcal{E}(\lambda+\lambda^{\prime})$ and thus, they define the same extension class.
\end{proof}
\begin{Lem}\label{trivial}
We have $b_{\lambda,R}=0$ if and only if $\lambda$ can be extended to a continuous homomorphism $\lambda\colon \GL_V(F)\to N.$
\end{Lem}
\begin{proof}
The class $b_{\lambda,R}$ is split if and only if there exists an element $(\overline{\Phi}, 1)\in \mathcal{E}(\lambda)$ such that $\overline{\Phi}$ is $\GL_V(F)$-invariant, i.e., for every $g\in \GL_V(F)$ there exists a constant $c(g)\in N$ such that
$$\Phi(gg^{\prime})=c(g)+\Phi(g^{\prime})$$
for all $g^{\prime}\in \GL_V(F)$.
We may assume that $\Phi(1)=0$.
But then $c(g)=\Phi(g)$ and, hence, $\Phi$ is a homomorphism that extends $\lambda$.
\end{proof}

\begin{Rem}
Note that the underlying $R[G]$-module of $\mathcal{E}(\lambda)$ does not depend on the topology on $R$.

One should view the representations $i_P(M_\lambda)$ as infinitesimal deformations of $i_P(R)$:
consider $R\oplus N$ as an $R$-algebra by putting $n_1 \cdot n_2 = 0$ for all $n_1,n_2\in N.$
For example, if $N$ is a free $R$-module of rank $s$, there exists an $R$-algebra isomorphism
$$R\oplus N\cong R[\epsilon_1,\ldots,\epsilon_s]/(\epsilon_1^{2},\ldots,\epsilon_s^{2}).$$
The map
\begin{align*}\Hom(F^{\times},N)&\too \left\{\chi\in \Hom(F^{\times},(R\oplus N)^{\times})\mid \chi \equiv 1 \bmod N\right\}\\
\lambda &\mapstoo \chi_{\lambda}=1 +\lambda
\end{align*}
is an isomorphism and the induction $i_P(M_\lambda)=i_P(\chi_{\lambda})$ is a module over $R\oplus N$.

\end{Rem}

\subsection{Homomorphisms}
We keep the notations from last section.
There is a canonical isomorphism
\begin{align}\label{isom}\Hom(P(F),N)\xrightarrow{\cong}\Hom(F^{\times},N)^{2},\quad \lambda\mapstoo (\lambda_1,\lambda_2),\end{align}
which is functorial in $N$:
every homomorphism from $P(F)$ to $N$ has to be trivial on the unipotent radical of $P(F)$.
Thus, it factors through the canonical map
$$P(F)\too \GL_{W_0}(F) \times \GL_{V\hspace{-0.2em}/W_0}(F).$$
Since $\SL_W(F)$ is the commutator subgroup of $\GL_W(F)$ for every finite-dimensional $F$-vector space $W$, every homomorphism
$$\lambda\colon\GL_{W_0}(F) \times \GL_{V\hspace{-0.2em}/W_0}(F) \too N$$
is of the form
$$\lambda(g_1,g_2)=\lambda_1(\det(g_1))+ \lambda_2(\det(g_2)) $$
for unique homomorphisms $\lambda_i\colon F^{\times} \to N.$
We will identify $\lambda$ with the pair $(\lambda_1,\lambda_2)$.

By the same argument every continuous homomorphism $\lambda\colon \GL_V(F)\to N$ is of the form $\lambda=\lambda^{\prime}\circ\det$ for a unique continuous homomorphism $\lambda^{\prime}\colon F^{\times}\to N$.
Therefore, Lemma \ref{trivial} implies the following:
\begin{Cor}
We have $b_{\lambda,R}=0$ if and only if $\lambda_1=\lambda_2.$
\end{Cor}
For a continuous group homomorphism $\lambda\colon F^{\times}\too N$ we define
$$c_{\lambda,R}=b_{(\lambda,0)}\in \Ext^{1}_{R[G],\cont}(R,v_{r}(N))\quad\mbox{and}\quad \mathcal{E}(\lambda)=\mathcal{E}(\lambda,0).$$
If $\lambda_1,\lambda_2\colon F^{\times}\too N$ are two group homomorphism, the corollary above implies that
$$b_{(\lambda_1,\lambda_2),R}=c_{\lambda_1-\lambda_2,R}.$$
The next claim follows immediately.
\begin{Cor}\label{injcor}
The map
$$\Hom(F^{\times},N)\too \Ext^{1}_{R[G],\cont}(R,v_{r}(N)),\quad \lambda\mapstoo c_{\lambda,R}$$
is an injective group homomorphism that is functorial in $R$ and $N$.
\end{Cor}

\subsection{Universality}\label{universal}
It is a natural question to ask whether there exists a class of topological $R$-modules such that the map in Corollary \ref{injcor} is an isomorphism.

In case $R=\Z$ one could reformulate the question as follows:
Let
$$c_{\univ}=c_{\id,\Z}\in \HH^{1}_{\cont}(G,v_{r}(F^{\times}))$$
be the extension associated with the identity $\id\colon F^{\times} \to F^{\times}.$
Functoriality implies that for every continuous homomorphism $\lambda\colon F^{\times}\to N$ the equality
$$\lambda_{\ast}(c_{\univ})=c_{\lambda,\Z} $$
holds in $\HH^{1}_{\cont}(G,v_{r}(F^{\times}))$.
Thus the question is whether $c_{\univ}$ is a universal extension.

We give a partial answer to this question.
Let $\widehat{F}^{\times}$ (resp.~$\widehat{\Z}$) be the profinite completion of $F^{\times}$ (resp.~$\Z$).
We define $\widehat{c}_{\univ}=c_{i,\widehat{\Z}},$ where $i\colon F^{\times}\to \widehat{F}^{\times}$ is the natural inclusion.
Let $N$ be an abelian profinite group.
Continuous homomorphisms from $F^{\times}$ to $N$ can be identified with continuous homomorphisms from $\widehat{F}^{\times}$ to $N$ and, by functoriality, we have
$$\lambda_{\ast}(\widehat{c}_{\univ})=c_{\lambda,\widehat{\Z}}$$
for every continuous homomorphism $\lambda\colon F^{\times }\to N$.
\begin{Def}
Let $N$ be a profinite group.
We say that $N$ is pretty good if $N$ is topologically finitely generated and every prime divisor $l$ of the pro-order of $N$ that is prime to $p$ is bon and banal for $\GL_V(F)$ in the sense of \cite{Dat}, Section 2.1.5.
\end{Def}
\begin{Pro}\label{puniv}
Let $F$ be a $p$-adic field.
For every abelian profinite group $N$ that is pretty good the homomorphism
$$\Hom(F^{\times},N)\too \Ext^{1}_{\widehat{\Z}[G],\cont}(\widehat{\Z},v_{r}(N)),\quad \lambda\mapstoo \lambda_\ast(\widehat{c}_{\univ})$$
is an isomorphism.
\end{Pro}
\begin{proof}
Let $0\leq t \leq s$ be integers.
It is enough to check that the map
\begin{align}\label{injmap}\Hom(F^{\times},\Z/l^{t} \Z)\too \Ext^{1}_{\Z/l^{s}\Z[\PGL_V(F)]}(\Z/l^{s}\Z,v_{r}(\Z/l^{t}\Z))\end{align}
surjects onto the space of smooth extensions for $l=p$ and every prime $l\neq p$ that is bon and banal.
The exact sequence
$$0 \too \Z/l^{t}\Z \xlongrightarrow{\cdot l^{s-t}}\Z/l^{s}\Z \too \Z/ l^{s-t} \Z \too 0$$
induces the exact sequence
$$0 \too v_r(\Z/l^{t}\Z) \too v_r(\Z/l^{s}\Z) \too v_r(\Z/ l^{s-t} \Z) \too 0.$$
We get the following commutative diagram with exact columns and injective horizontal maps:
\begin{center}
 \begin{tikzpicture}
    \path 	(0,0) node[name=A]{$0$}
		(0,-1) node[name=B]{$\Hom(F^{\times},\Z/l^{t}\Z)$}
		(0,-2.2) node[name=C]{$\Hom(F^{\times},\Z/l^{s}\Z)$}
		(0,-3.4) node[name=D]{$\Hom(F^{\times},\Z/l^{s-t}\Z)$}
		(6,0) node[name=F]{$0$}
		(6,-1) node[name=G]{$\Ext^{1}_{\Z/l^{s}\Z[\PGL_V(F)]}(\Z/l^{s}\Z,v_{r}(\Z/l^{t}\Z))$}
		(6,-2.2) node[name=H]{$\Ext^{1}_{\Z/l^{s}\Z[\PGL_V(F)]}(\Z/l^{s}\Z,v_{r}(\Z/l^{s}\Z))$}
		(6,-3.4) node[name=I]{$\Ext^{1}_{\Z/l^{s}\Z[\PGL_V(F)]}(\Z/l^{s}\Z,v_{r}(\Z/l^{s-t}\Z))$};
    \draw[->] (A) -- (B) ;
		\draw[->] (B) -- (C) ;
		\draw[->] (C) -- (D) ;
		\draw[->] (F) -- (G) ;
		\draw[->] (G) -- (H) ;
		\draw[->] (H) -- (I) ;
    \draw[right hook->] (B) -- (G) ;
    \draw[right hook->] (C) -- (H) ;
    \draw[right hook->] (D) -- (I) ;
  \end{tikzpicture} 
\end{center}
The exactness of the second column follows from Lemma \ref{invariants}.
By a simple diagram chase we see that it is enough to prove the claim above in the case $s=t$.

The case $l\neq p$: By assumption $l$ does not divide the order of the torsion subgroup of $F^{\times}$.
In particular, $\Hom(F^{\times},\Z/l^{s} \Z)$ is a free $\Z/l^{s} \Z$-module of rank $1$.
By \cite{Dat}, Theorem 1.3 respectively \cite{Orlik}, Theorem 1, the space of smooth extension is also free of rank $1$.
Thus, the claim follows from the injectivity of \eqref{injmap}.

The case $l=p$: By \cite{CDHN}, Theorem 1.10 (2), there exists an isomorphism between $\Hom(F^{\times},\Z/p^{s}\Z)$ and the space of continuous extensions.
The claim now follows by injectivity of \eqref{injmap} and the finiteness of $\Hom(F^{\times},\Z/p^{s}\Z)$.
\end{proof}

\begin{Rem}
We expect that Proposition \ref{puniv} also holds in the case that $F$ is a local function field.
More precisely, we expect that the map \eqref{injmap} for $l=p$ is inverse to the one constructed in \cite{CDHN}.
\end{Rem}

\section{Invertible analytic functions on Drinfeld symmetric spaces}\label{2}
Using the main theorem of \cite{Gekeler} we will prove that the group of invertible analytic functions on Drinfeld's upper half space is isomorphic to a dual of the universal extension (for r=n-1) defined in Section \ref{universal}.

\subsection{Zero cycles}
Let us recall that Drinfeld's upper half space $\Omega=\Omega_V$ of dimension $n-1$ is the complement of all $F$-rational hyperplanes in $\PP(V)$, i.e.:
 $$\Omega=\PP(V)\setminus \bigcup_{H\subsetneq V} \PP(H).$$
It is a rigid analytic variety over $F$ on which the group $G=\PGL_V(F)$ acts naturally.
Let $\C_F$ be the completion of an algebraic closure of $F$ with respect to the unique extension of the norm.
Let $\mathcal{Z}_0(\Omega_{\C_F})=\Z[\Omega(\C_F)]$ be the free abelian group on the $\C_F$-valued points of $\Omega$.
We define $\mathcal{Z}_0^0(\Omega_{\C_F})$ as the kernel of the degree map
$$\deg\colon \mathcal{Z}_0(\Omega_{\C_F})\too \Z,\quad \sum_{x}a_x [x] \mapstoo \sum a_x.$$
In the following by an extension $E/F$ we always mean a closed subextension of $\C_F/F$ and we write $\iota_E\colon F \into E$ for the inclusion.
For such an extension $E$ we put
\begin{align*}
\mathcal{Z}_0(\Omega_{E})&=\mathcal{Z}_0(\Omega_{\C_F})^{\Aut(\C_F,E)}\\
\intertext{and}
\mathcal{Z}_0^0(\Omega_{E})&=\mathcal{Z}_0^{0}(\Omega_{\C_F})^{\Aut(\C_F,E)}.
\end{align*}
The support of a zero cycle $z\in \mathcal{Z}_0(\Omega_{E})$ is always defined over a finite extension of $E$.
Let $c_{\geom}(E)\in \Ext^1_{\Z[G]}(\mathcal{Z}_0^0(\Omega_{E}),\Z)$ be the class of the exact sequence
$$0 \too \mathcal{Z}_0^0(\Omega_{E}) \too \mathcal{Z}_0(\Omega_{E})\xlongrightarrow{\deg}\Z\too 0.$$

\subsection{Zero cycles of degree 0 and Steinberg representations}
We put $V^{\ast}=\Hom_F(V,F)$.
Let $z=\sum_x a_x [x]\in \mathcal{Z}_0^0(\Omega_{E})$ be a zero cycle of degree $0$.
There exist lifts $v_x\in V\otimes_F \C_F \setminus\left\{0\right\}$ of the elements $x$ such that the formal sum $\sum_x a_x [v_x]$ is invariant under $\Aut(\C_F,E)$.
By definition of $\Omega$ we have $\ell(v_x)\neq 0$ for all non-zero elements $\ell\in V^{\ast}$.
Since the cycle $z$ is of degree $0$, the function
$$\widetilde{\Psi^0}(z)\colon V^{\ast}\setminus\left\{0\right\}\too E^{\times},\quad \ell\mapstoo \prod_x \ell(v_x)^{a_x}$$
descends to a map
$$\widetilde{\Psi^0}(z)\colon \PP_{V^{\ast}}(F)\too E^{\times}.$$
Choosing different lifts $v_x$ changes $\widetilde{\Psi^0}(z)$ only up to a constant and, therefore, the induced element $\Psi(z)\in v_{1,V^{\ast}}(E^{\times})$ does not depend on the chosen lifts.
The resulting map
$$\Psi^0\colon\mathcal{Z}_0^0(\Omega_{E})\too v_{1,V^{\ast}}(E^{\times}),\quad z\mapstoo \Psi^0(z) $$
is $G$-equivariant.
Here and in the following we always identify $\GL_V$ and $\GL_{V^{\ast}}$ via the $\GL_V$-action on $V^{\ast}$ given by $(g.\ell)(v)=\ell(g^{-1}(v))$.
Given a topological abelian group we also abbreviate $v(N)=v_{1,V^{\ast}}(N)\cong v_{n-1,V}(N).$

\subsection{Zero cycles and the universal extension}
We fix an element $y_0\in\PP_{V^{\ast}}(F)$ and denote by $P\subseteq \GL_{V^{\ast}}\cong \GL_{V}$ its stabilizer.
Let $c_{\univ}$ be the universal extension of $\Z$ by $v(F^{\times})$ associated with the identity $\id\colon F^{\times}\to F^{\times}$ and the parabolic subgroup $P$ as in Section \ref{universal}.

The following is a generalization of \cite{BG2}, Lemma 6.8, from the case $n=2$ to arbitrary dimension.
\begin{Pro}\label{comparison}
For every extension $E/F$ the equality
$$\Psi^0_{\ast}(c_{\geom}(E))=\iota_{E,\ast}(c_{\univ})$$
holds in $\HH^{1}(G,v(E^{\times})).$
\end{Pro}
\begin{proof}
We fix a lift $\ell_0\in V^{\ast}$ of $y_0$.
Under the identification \eqref{isom} we have
\begin{align}\label{mult}p.\ell_0=(\id,1)(p)\cdot \ell_0\end{align}
for all $p\in P(F)$.
For $x\in\Omega(\C_F)$ we choose a lift $v_x\in V_{\C_F}$ and define the function
$$\Phi_x\colon \GL_{V^{\ast}}(F)\too \C_F^{\times},\quad g\mapstoo \frac{(g.\ell_0)(v_x)}{\ell_0(v_x)}.$$
The function is independent of the choices of lifts of $y_0$ and $x$ and, by \eqref{mult}, fulfils
$$\Phi_x(gp)=\Phi_x(g)\cdot (\id,1)(p)$$
for all $g\in \GL_{V^{\ast}}(F)$ and $p\in P(F)$.
We thus get a well-defined map
$$\Psi\colon \mathcal{Z}_0(\Omega_{E})\too \mathcal{E}(\iota_E),\quad \sum_x a_x[x]\mapstoo (\prod_x (\overline{\Phi_x})^{a_x},\sum_x a_x).$$
Note that a priori $\Psi$ takes values in $\mathcal{E}(\iota_{\C_F})$ but one can argue as before that it factors through $\mathcal{E}(\iota_E)$.

For all $g,g^{\prime} \in \GL_{V^{\ast}(F)}$ and all $x\in\Omega(E)$ we have
\begin{align*}
\Psi(g^{\prime}.[x])(g)
=\frac{(g.\ell_0)(g^{\prime}.v_x)}{\ell_0(g^{\prime}.v_x)}
=\frac{(g^{\prime}g.\ell_0)(v_x)}{\ell_0(v_x)} \cdot \frac{\ell_0(v_x)}{(g^{\prime}.\ell_0)(v_x)}
= (g^{\prime}.\Psi(g))(x) \cdot k_{x,g^{\prime}},
\end{align*}
where $k_{x,g^{\prime}}$ is a constant that does not depend on $g$.
Thus, the homomorphism $\Psi$ is $\GL_{V^{\ast}}(F)$-equivariant.

Similarly, for $x, x^{\prime}\in \Omega(E)$ we have
$$\Psi([x]-[x^{\prime}])(g)=\frac{(g.\ell_0)(v_x)}{\ell_0(v_x)} \frac{\ell_0(v_{x^{\prime}})}{(g.\ell_0)(v_{x^{\prime}})}=\frac{(g.\ell_0)(v_x)}{(g.\ell_0)(v_{x^{\prime}})}\cdot k_{x,x^{\prime}}$$
where $k_{x,x^{\prime}}$ is a constant independent of $g$.
Thus, the restriction of $\Psi$ to cycles of degree $0$ agrees with $\Psi^{0}$, i.e., the diagram
\begin{center}
 \begin{tikzpicture}
    \path 	(0.5,0) node[name=A]{$0$}
		(3,0) node[name=B]{$\mathcal{Z}_0^0(\Omega_{E})$}
		(6,0) node[name=C]{$\mathcal{Z}_0(\Omega_{E})$}
		(8.5,0) node[name=D]{$\Z$}
		(10.5,0) node[name=E]{$0$}
		(0.5,-1.5) node[name=F]{$0$}
		(3,-1.5) node[name=G]{$v(E^{\times})$}
		(6,-1.5) node[name=H]{$\mathcal{E}((\iota_E,1))$}
		(8.5,-1.5) node[name=I]{$\Z$}
		(10.5,-1.5) node[name=J]{$0$};
    \draw[->] (A) -- (B) ;
		\draw[->] (B) -- (C) ;
		\draw[->] (C) -- (D) ;
		\draw[->] (D) -- (E) ;
		\draw[->] (F) -- (G) ;
		\draw[->] (G) -- (H) ;
		\draw[->] (H) -- (I) ;
		\draw[->] (I) -- (J) ;
    \draw[->] (B) -- (G) node[midway, right]{$\Psi^0$};
    \draw[->] (C) -- (H) node[midway, right]{$\Psi$};
    \draw[->] (D) -- (I) node[midway, right]{$=$};
  \end{tikzpicture} 
\end{center}
of $\Z[G]$-modules is commutative and, therefore, the claim follows.
\end{proof}

\subsection{Gekeler's theorem}
For an extension $E/F$ we write $\mathcal{A}_E=\mathcal{O}_{\Omega_E}(\Omega_{E})$ for the ring of rigid analytic functions on $\Omega$ that are defined over $E$.
If $N$ is any abelian group, we identify $\Hom_\Z(\mathcal{Z}_0(\Omega_{\C_F}),N)$ with the space $\Maps(\Omega(\C_F),N)$ of all (set theoretic) functions from $\Omega(\C_F)$ to $N$.
Choosing a base point $x^{\prime}\in \Omega(\C_F)$ we also get a $\GL_{V}(F)$-equivariant isomorphism
$$\Hom_\Z(\mathcal{Z}_0^{0}(\Omega_{\C_F}),N)\too \Maps(\Omega(\C_F),N)/N,\quad f\mapstoo f([x]-[x^{\prime}])$$
that is independent of the choice of $x^{\prime}.$
Thus, by taking $\Hom_\Z(\cdot,\C_F^{\times})$ the map $\Psi^{0}$ induces the homomorphism
$$(\Psi^{0})^{\ast}\colon \Hom_\Z(v(\C_F^{\times}),\C_F^{\times}) \too \Maps(\Omega(\C_F),\C_F^{\times})/\C_F^{\times}.$$
Precomposing with the map
$$\integ\colon\Hom_\Z(v(\Z),\Z)\too \Hom_\Z(v(\C_F^{\times}),\C_F^{\times})$$
defined by the integration pairing in Section \ref{Integration} yields the homomorphism
$$\Xi^{0}=(\Psi^{0})^{\ast}\circ\integ\colon \Hom_\Z(v(\Z),\Z) \too \Maps(\Omega(\C_F),\C_F^{\times})/\C_F^{\times}.$$
By an easy argument with Riemann sums (or rather Riemann products) we see that the map takes values in $\mathcal{A}_{\C_F}^{\times}/\C_F^{\times}.$
A rationality argument as before shows that the map $\Xi^{0}$ factors through $\mathcal{A}_{F}^{\times}/F^{\times}$.
For an extension $E/F$ let
$$\Xi^{0}_E\colon \Hom_\Z(v(\Z),\Z) \too\mathcal{A}_E^{\times}/E^{\times}$$
be the induced homomorphism.

\begin{Thm}\label{Gekeler}
For every extension $E/F$ the map
$$\Xi^{0}_E\colon \Hom_\Z(v(\Z),\Z) \too \mathcal{A}_E^{\times}/E^{\times}$$
is a $\GL_V(F)$-equivariant isomorphism.
In particular, the $\GL_V(F)$-module $\mathcal{A}_E^{\times}/E^{\times}$ does not depend on $E$.
\end{Thm}
\begin{proof}
For every extension $E$ the canonical map
$$\mathcal{A}_E^{\times}/E^{\times}\too \mathcal{A}_{\C_F}^{\times}/\C_F^{\times}$$
is injective.
Thus, it is enough to prove that
$$\Xi^{0}\colon \Hom_\Z(v(\Z),\Z) \too \mathcal{A}_{\C_F}^{\times}/\C_F^{\times}$$
is an isomorphism.
But this follows directly from the proof of \cite{Gekeler}, Theorem 3.11.
\end{proof}

\subsection{The main theorem}
Let $E/F$ be an extension.
Applying $\Hom_\Z(\cdot,E^\times)$ to the exact sequence
$$0 \too v(E^{\times})\too \mathcal{E}(\iota_E) \too \Z\too 0$$
induces the exact sequence 
$$ 0 \too E^{\times}\too \Hom_\Z(\mathcal{E}(\iota_E),E^{\times})\too \Hom_\Z(v(E^{\times}),E^{\times})\too 0.$$
Let $\mathcal{E}_{\univ,E}^{\vee}$ be the pullback of this extension
along $\integ$, i.e., we have an exact sequence
$$0 \too E^{\times}\too \mathcal{E}_{\univ,E}^{\vee}\too \Hom_\Z(v(\Z),\Z)\too 0.$$

\begin{Thm}\label{mainthm}
For every extension $E/F$ there is a unique $\PGL_V(F)$-equivariant isomorphism
$$\Xi_E\colon \mathcal{E}_{\univ,E}^{\vee} \too \mathcal{A}_E^{\times}$$
such that the following diagram commutes:
\begin{center}
 \begin{tikzpicture}
    \path 	(10.8,-1.5) node[name=A]{$0$}
		(8.5,-1.5) node[name=B]{$\mathcal{A}_E^{\times}/E^{\times}$}
		(5.5,-1.5) node[name=C]{$\mathcal{A}_E^{\times}$}
		(3,-1.5) node[name=D]{$E^{\times}$}
		(1.3,-1.5) node[name=E]{$0$}
		(10.8,0) node[name=F]{$0$}
		(8.5,0) node[name=G]{$\Hom(v(\Z),\Z)$}
		(5.5,0) node[name=H]{$\mathcal{E}_{\univ,E}^{\vee}$}
		(3,0) node[name=I]{$E^{\times}$}
		(1.3,0) node[name=J]{$0$};
    \draw[->] (B) -- (A) ;
		\draw[->] (C) -- (B) ;
		\draw[->] (D) -- (C) ;
		\draw[->] (E) -- (D) ;
		\draw[->] (G) -- (F) ;
		\draw[->] (H) -- (G) ;
		\draw[->] (I) -- (H) ;
		\draw[->] (J) -- (I) ;
    \draw[->] (G) -- (B) node[midway, right]{$\Xi^{0}_E$};
    \draw[->] (H) -- (C) node[midway, right]{$\Xi_E$};
    \draw[->] (I) -- (D) node[midway, right]{$=$};
  \end{tikzpicture} 
\end{center}

\end{Thm}
\begin{proof}
As before, we only treat the case $E=\C_F$.
Applying $\Hom_\Z(\cdot,\C_F^{\times})$ to the commutative diagram in the proof of Lemma \ref{comparison} (with $E=\C_F$) yields the following commutative diagram with exact rows:
\begin{center}
 \begin{tikzpicture}
    \path 	(11.8,-1.5) node[name=A]{$0$}
		(9.3,-1.5) node[name=B]{$\Maps(\Omega(\C_F),\C_F^{\times})/\C_F^{\times}$}
		(5.7,-1.5) node[name=C]{$\Maps(\Omega(\C_F),\C_F^{\times})$}
		(3,-1.5) node[name=D]{$\C_F^{\times}$}
		(1.2,-1.5) node[name=E]{$0$}
		(11.8,0) node[name=F]{$0$}
		(9.3,0) node[name=G]{$\Hom_\Z(v,\C_F^{\times})$}
		(5.7,0) node[name=H]{$\Hom_\Z(\mathcal{E}(\iota_{\C_F}),\C_F^{\times})$}
		(3,0) node[name=I]{$\C_F^{\times}$}
		(1.2,0) node[name=J]{$0$};
    \draw[->] (B) -- (A) ;
		\draw[->] (C) -- (B) ;
		\draw[->] (D) -- (C) ;
		\draw[->] (E) -- (D) ;
		\draw[->] (G) -- (F) ;
		\draw[->] (H) -- (G) ;
		\draw[->] (I) -- (H) ;
		\draw[->] (J) -- (I) ;
    \draw[->] (G) -- (B) node[midway, right]{$(\Psi^{0})^{\ast}$};
    \draw[->] (H) -- (C) node[midway, right]{$\Psi^{\ast}$};
    \draw[->] (I) -- (D) node[midway, right]{$=$};
  \end{tikzpicture} 
\end{center}
Thus, by construction there exists a map $\Xi_{\C_F}\colon  \mathcal{E}_{\univ,\C_F}^{\vee} \too \Maps(\Omega(\C_F),\C_F^{\times}).$
We know that modulo constants the map takes values in analytic functions.
Therefore, $\Xi _{\C_F}$ itself takes values in analytic functions, and the claim follows from Theorem \ref{Gekeler}.
\end{proof}

Let us end this section by giving an explicit description of the ``$p$-adic completion'' of $\mathcal{E}_{\univ,F}^{\vee}$ in the case that $F$ is a $p$-adic field:
let $\widetilde{F}^{\times}$ be the torsion-free part of the pro-$p$ completion of $F^{\times}$.
It is a free $\Z_p$-module of rank $[F:\Q_p]+1=d+1.$
Write $i\colon F^{\times}\to \widetilde{F}^{\times}$ for the natural map and $\mathcal{E}(i)$ for the associated extension of $\Z_p$ by $v(\widetilde{F}^{\times}).$
Let $\mathcal{E}_{\univ}^{\vee,p}$ be the pullback of $\Hom_{\Z_p}(\mathcal{E}(i),\widetilde{F}^{\times})$ along
$$\Hom_{\Z_p}(v(\Z_p),\Z_p)\too \Hom_{\Z_p}(v(\widetilde{F}^{\times}),\widetilde{F}^{\times}).$$
Thus, we have an exact sequence
$$0\too \widetilde{F}^{\times}\too \mathcal{E}_{\univ}^{\vee,p} \too \Hom_{\Z_p}(v(\Z_p),\Z_p)\too 0.$$
Choosing a basis $\left\{\lambda_1,\ldots\lambda_{d+1}\right\}$ of the free $\Z_p$-module $\Hom(F^{\times},\Z_p)$ we get an isomorphism
$$\mathcal{E}_{\univ}^{\vee,p}\cong\Hom_{\Z_p}(\mathcal{E}(\lambda_1)\oplus_{v(\Z_p)}\cdots \oplus_{v(\Z_p)}\mathcal{E}(\lambda_{d+1}),\Z_p),$$
where $\mathcal{E}(\lambda_i)$ is the extension of $\Z_p$ by $v(\Z_p)$ associated with $\lambda_i$.

\section{Lifting obstructions of theta cocycles}
The aim of this section is to apply our main result to the study of lifting obstructions of theta cocycles.
In \cite{DV} Darmon and Vonk initiated the theory of rigid meromorphic cocycles, i.e., elements in the cohomology group $\HH^1(\SL_2(\Z[1/p]),\mathcal{M}^{\times})$, where $\mathcal{M}^{\times}$ is the group of invertible meromorphic functions on Drinfeld's $p$-adic upper half plane $\Omega$.
They provide a large supply of classes in the space of theta cocycles, i.e.,  elements of $\HH^1(\SL_2(\Z[1/p]),\mathcal{M}^{\times}/\C_p^{\times})$ (see also \cite{Ge5} and \cite{GMX} for a generalization of the theory to other number fields and congruence subgroups).
It is thus a natural question to ask whether these classes can be lifted to genuine meromorphic cocycles.

In the following we want to show that for rigid analytic theta cocycles, i.e., classes in $\HH^1(\PGL_2(\Z[1/p]),\mathcal{A}_{\Q_p}^{\times}/\Q_p^{\times})$ the answer is often negative.
But one may still lift these classes to elements in $\HH^1(\SL_2(\Z[1/p]),\mathcal{A}_{\Q_p}^{\times}/\Lambda)$, where $\Lambda \subseteq \Q_p^{\times}$ is a discrete subgroup that can be computed in terms of Galois representations.
In fact, we give general results for cuspidal analytic theta cocycles for Hilbert modular groups.
In addition, we explain that our methods also yield a new proof of a recent result of Darmon--Pozzi--Vonk on the Dedekind--Rademacher cocycle (cf.~\cite{DPV2}, Theorem A).
We end this note by a generalization of the whole story to higher rank unitary groups.

We will use the following notation throughout this section:
If $M$ is an abelian group, we denote its $\Z$-dual by $M^\ast=\Hom_\Z(M,\Z)$.

\subsection{$\LI$-invariants of Galois representations}
Let $F$ be a $p$-adic field with absolute Galois group $\G_F$.
Let $\rho\colon \G_F\to \GL_2(\Q_p)$ be a Galois representation that is an extension of $\Q_p$ by $\Q_p(1)$, i.e., it defines a class $[\rho]$ in $\HH^{1}(\G_F,\Q_p(1)).$
Local class field theory gives an isomorphism
$$\Hom(F^{\times},\Q_p)\cong\HH^{1}(\G_F,\Q_p).$$
We define the $\LI$-invariant
$$\LI(\rho)\subseteq \Hom(F^{\times},\Q_p)$$
of $\rho$ as the orthogonal complement of $[\rho]$ under the local Tate pairing
$$\HH^{1}(\G_F,\Q_p) \times \HH^{1}(\G_F,\Q_p(1))\xlongrightarrow{\cup} \HH^{2}(\G_F,\Q_p(1))\cong \Q_p. $$
Since the pairing is non-degenerate, $\LI(\rho)$ is a subspace of codimension at most one.
Its codimension is one if and only if $\rho$ is non-split.

Suppose the $p$-adic valuation $\ord_F\colon F^{\times}\too \Z$ is not an element of $\LI(\rho)$.
Equivalently, $\rho$ is not crystalline.
Then the subgroup
$$\Lambda_\rho=\left\{q\in F^{\times}\mid \lambda(q)=0\ \forall \lambda \in \LI(\rho)\right\}\subseteq F^{\times}$$
is discrete.
Moreover, it is a finitely generated abelian group of rank one.

\subsection{Cuspidal theta cocycles for Hilbert modular groups}
Let $F$ be a totally real field of degree $d$ and with ring of integers $\mathcal{O}_F$.
We assume that $F$ has narrow class number one.
(One can drop this assumption by formulating everything in the adelic language as in \cite{Sp} or \cite{Ge2}.
For ease of exposition we stick to the case of narrow class number one.)
We fix a prime $\p$ of $F$ lying above the rational prime $p$.
We write $F_\p$ for the completion of $F$ at $\p$,  $\mathcal{O}_\p$ for the valuation ring of $F_\p$ with local uniformizer $\varpi$ and $q$ for the cardinality of the residue field of $\mathcal{O}_F$.
We put
\begin{align*}
G_\p&=\PGL_2(F_\p),\quad K_\p=\PGL_2(\mathcal{O}_\p)\\
\intertext{and}
I_\p&=\left\{ k\in K_\p \mid k\ \mbox{upper triangular}\ \bmod \p \right\}
\end{align*}
The matrix
$$w=\begin{pmatrix} 0 & 1 \\ \varpi & 0 \end{pmatrix}$$
normalizes the Iwahori subgroup $I_\p$.
We define $\widetilde{I}_\p$ to be the subgroup generated by $I_\p$ and $w$ and let
$$\chi_\p\colon \widetilde{I}_\p \too \{\pm 1\}$$
be the unique non-trivial homomorphism that is trivial on $I_\p$.

Let $f$ be a cuspidal Hilbert newform with trivial nebentypus and parallel weight $2$.
We assume that $f$ is Steinberg at $\p$, i.e., $\p$ divides the level of $f$ exactly once and $U_\p f= f.$
For simplicity, we assume that all Hecke eigenvalues of $f$ are rational.
In that case one can attach to $f$ an elliptic curve $E_f/F$ that has split multiplicative reduction at $\p$.
Let $\rho_f\colon \G_F \to \GL_2(\Q_p)$ be the associated Galois representation.
By Tate's $p$-adic uniformization theorem, the restriction $\rho_{f,\p}$ of $\rho_f$ to a decomposition group at $\p$ is a non-crystalline extension of $\Q_p$ by $\Q_p(1)$.
Thus, we can define the discrete subgroup $\Lambda_{\rho_{f,\p}}\subseteq F_\p^{\times}.$

We put $\St(\Z)=C(\PP^{1}(F_\p),\Z)/\Z$ and for any non-zero ideal $\n$ of $\mathcal{O}_F$ we define
\begin{align*}
\Gamma_0(\n)&=\left\{\gamma\in \PGL_2(\mathcal{O}_F)^{+}\mid \gamma \mbox{ is upper triangular modulo } \n\right\},\\
\intertext{respectively}
\Gamma_0(\n)^{\p}&=\left\{\gamma\in \PGL_2(\mathcal{O}_F[1/\p])^{+}\mid \gamma \mbox{ is upper triangular modulo } \n\right\}
\end{align*}
where the superscript $+$ denotes matrices with totally positive determinant.

Suppose that $\p$ divides $\n$ exactly once.
Then, by \cite{Sp}, Proposition 5.8 (b), the map
$$\HH^{\ast}(\Gamma_0(\n)^{\p}, \St(\Z)^\ast)\otimes \Q \too \HH^{\ast}(\Gamma_0(\n),\Q)$$
given by evaluation at a non-zero Iwahori-fixed vector induces an isomorphism on $f$-isotypic components for the Hecke algebra away from $\p$.
In particular, the $f$-isotypic component vanishes for $i\neq d$ and is non-zero for $i=d$ and large enough level.
Let us give a sketch of the proof:
first let us remind ourselves of the definition of a (compactly) induced module.
Given a group $V$, a subgroup $U\subseteq V$ and an $\Z[U]$-module $A$ the induction $\cind_{U}^{V}A$ of $A$ to $V$ is the space
all functions $f\colon V \to A$ such that
\begin{itemize}
\item $f(vu)=u^{-1}f(v)$ for all $u\in U$, $v\in V$ and
\item $f$ has finite support modulo $U$.
\end{itemize}
the group $V$ acts on $\cind_{U}^{V}A$ via left translation.
The $I_\p$-invariants of $\St_\p(\Z)$ are a free $\Z$-module of rank $1$ that generate $\St_\p$ as a $G_\p$-module.
The matrix $w$ acts on it by $\chi_\p$.
Hence Frobenius reciprocity induces a surjective $G_\p$-equivariant homomorphism
$$\cind_{\widetilde{I}_\p}^{G_\p}\chi_\p \ontoo \St_\p(\Z).$$
In fact, this surjection fits into a short exact sequence
\begin{align}\label{Steinbergresolution}
0 \too \cind_{K_\p}^{G_\p}\chi_\p \too \cind_{\widetilde{I}_\p}^{G_\p}\chi_\p \too \St_\p(\Z) \too 0
\end{align}
of $G_\p$-modules that identifies the Steinberg representation with the first cohomology with compact support of the Bruhat--Tits tree (see for example \cite{Sp}, equation (18)).
Utilizing Shapiro's Lemma the short exact sequence above induces a long exact sequence in cohomology of the form
$$ \ldots\to \HH^{i}(\Gamma_0(\n)^{\p}, \St(\Z)^\ast)\to \HH^{i}(\Gamma_0(\n),\Z)^{W_\p=-1} \to \HH^{i}(\Gamma_0(\n\p^{-1}),\Z)\to \ldots$$
where $W_\p$ denotes the Atkin--Lehner operator at $\p$.
Since $\p$ divides the conductor of $f$, we see that
$$(\HH^{i}(\Gamma_0(\n\p^{-1})\otimes \Q)^f=0,$$
where the superscript $f$ denotes taking the $f$-isotypic component.
Thus, the claim follows.
Let us mention that the above argument also yields the following:
if $g$ is a cuspidal Hecke eigenform that is not Steinberg at $p$, then we have
$$(\HH^{i}(\Gamma_0(\n)^{\p}, \St(\Z)^\ast)\otimes \Q)^g=0$$
for all $i\geq 0$.

Let $\mathcal{A}$ denote the ring of analytic functions on Drinfeld's $p$-adic upper half plane over $F_\p$.
By Theorem \ref {Gekeler}, which in this case is due to van der Put (see \cite{vdP}, Proposition 1.1), we have
$$(\HH^{i}(\Gamma_0(\n)^{\p}, \mathcal{A}^{\times}/F_\p^{\times})\otimes \Q)^f =0\ \mbox{for}\ i\neq d $$
Moreover, we have
$$(\HH^{d}(\Gamma_0(\n)^{\p}, \mathcal{A}^{\times}/F_\p^{\times})\otimes \Q)^f \neq 0$$
if the level $\n$ is large enough. 

By our main theorem the following diagram is commutative:
\begin{center}
 \begin{tikzpicture}
    \path 	
		(9.3,-1.5) node[name=G]{$(\HH^{d+1}(\Gamma_0(\n)^{\p}, F_\p^{\times})\otimes \Q)^f$}
		(3,-1.5) node[name=H]{$(\HH^{d}(\Gamma_0(\n)^{\p}, \mathcal{A}^{\times}/F_\p^{\times})\otimes \Q)^f$}
		(9.3,-3) node[name=B]{$(\HH^{d+1}(\Gamma_0(\n)^{\p}, F_\p^{\times})\otimes \Q)^f$}
		(3,-3) node[name=C]{$(\HH^{d}(\Gamma_0(\n)^{\p}, \St(\Z)^\ast)\otimes \Q)^f$};
		\draw[->] (C) -- (B) ;
		\draw[->] (H) -- (G) ;
    \draw[->] (G) -- (B) node[midway, right]{$=$};
    \draw[->] (H) -- (C) node[midway, left]{$\cong$};
  \end{tikzpicture} 
\end{center}
where the horizontal maps are the boundary maps induced by the short exact sequences
$$0 \too F_\p^{\times} \too \mathcal{A}^{\times} \too \mathcal{A}^{\times}/F_\p^{\times} \too 0$$
and
$$0 \too  F_\p^{\times} \too \mathcal{E}_{\univ,F_\p}^{\vee} \too \St(\Z)^\ast \too 0.$$
By \cite{Sp}, Lemma 6.2 (b), the composition of the lower horizontal map with the homomorphism
$$\HH^{d+1}(\Gamma_0(\n)^{\p}, F_\p^{\times})\otimes \Q)^f\xlongrightarrow{\ord_\p} \HH^{d+1}(\Gamma_0(\n)^{\p}, \Z)\otimes \Q)^f$$
induced by the $p$-adic valuation is an isomorphism.
This should be seen as the automorphic analogue of the fact that the local Galois representation $\rho_{f,\p}$ is non-crystalline.
Again we give a short sketch of the proof:
The map under consideration
\begin{align}\label{ord}
\HH^{d}(\Gamma_0(\n)^{\p}, \St(\Z)^\ast)\otimes \Q)^f\too \HH^{d+1}(\Gamma_0(\n)^{\p}, \Z)\otimes \Q)^f
\end{align}
is the boundary map of the long exact sequence induced by the short exact sequence
\begin{align}\label{ordext}
0 \too  \Z \too \mathcal{E}(\ord_\p)^\ast \too \St(\Z)^\ast\too 0
\end{align}
attached to the homomorphism $\ord_\p\colon F_\p^\times \to \Z$.
By \cite{Sp}, Lemma 3.11 (c), 
there exists an exact sequence of $\PGL_2(F_\p)$-modules of the form
\begin{align}\label{Eisensteinpart}
0\to\cind_{K_\p}^{G_\p}\xlongrightarrow{T_\p-q-1}\cind_{K_\p}^{G_\p}\too \mathcal{E}(\ord_\p)\to 0.
\end{align}
As before, we have
$$\HH^{i}(\Gamma_0(\n\p^{-1}),\Q)^f=0$$
for all $i\geq 0$.
Thus, via the long exact sequence induced by the short exact sequence above we also see that
$$\HH^{i}(\Gamma_0(\n)^{\p},\mathcal{E}(\ord_\p)^\ast) \otimes \Q)^f  =0$$
holds for all $i$, which proves that \eqref{ord} is an isomorphism.

Let  $\widehat{\Gamma}_0(\n)^{\p}\subseteq \PGL_2(\mathcal{O}_F[1/\p])$ be the subgroup given by the same congruence conditions as $\Gamma_0(\n)^{\p}$ but without the positivity condition.
The quotient group
$$\Sigma=\widehat{\Gamma}_0(\n)^{\p}/\Gamma_0(\n)^{\p}\cong\{\pm 1\}^{d}$$
naturally acts on all the cohomology groups considered above.
For any character $\epsilon\colon \Sigma\to \{\pm 1\}$ the $\epsilon$-isotypic component
$$(\HH^{d}(\Gamma_0(\n)^{\p}, \St(\Z)^\ast)\otimes \Q)^{f,\epsilon}$$
is an irreducible module over the Hecke algebra away from $\p$.
Hence, the above result by Spie\ss~implies that there exists a discrete subgroup 
$$\Lambda_{f,\p}^{\epsilon}\subseteq F_\p^\times$$
of rank one, which is unique up to homothety, such that the image of the $\epsilon$-component under the lower horizontal map in the commutative diagram above is equal to
$$(\HH^{d+1}(\Gamma_0(\n)^{\p}, \Lambda_{f,\p})\otimes \Q)^{f,\epsilon}.$$
The equality of automorphic and Fontaine--Mazur $\LI$-invariants, which in this case is Theorem 4.1 of \cite{GeR} respectively Theorem 3.7 of \cite{Sp3},
implies that we may take
$$\Lambda_{f,\p}^\epsilon=\Lambda_{\rho_{f},\p}.$$
Combining all the results above we have proven the following:
\begin{Pro}
The canonical map
$$(\HH^{d}(\Gamma_0(\n)^{\p}, \mathcal{A}^{\times}/\Lambda_{\rho_{f,\p}})\otimes \Q)^f \too (\HH^{d}(\Gamma_0(\n)^{\p}, \mathcal{A}^{\times}/F_\p^{\times})\otimes \Q)^f$$
is an isomorphism.
Moreover, we have
$$(\HH^{d}(\Gamma_0(\n)^{\p}, \mathcal{A}^{\times})\otimes \Q)^f=0.$$
\end{Pro}

\begin{Rem}
The same reasoning implies the result for forms of parallel weight $2$ on inner twists of $\PGL_2(F)$ that are split at $\p$.

In case $F=\Q$ the equality of automorphic and Fontaine--Mazur $\LI$-invariants was already known by work of Darmon (cf.~\cite{D}, Theorem 1) if $f$ corresponds to an elliptic curve, and in general by work of Dasgupta (cf.\cite{Das}, Proposition 5.20).
\end{Rem}

\subsection{An accidental $\LI$-invariant}
In the previous section we always assumed that the Hecke eigenform under consideration is cuspidal.
In this section we treat the Eisenstein case for $F=\Q.$
We abbreviate $\Gamma^p=\PGL_2(\Z[1/p])$.

Let $E_2(p)$ be the Eisenstein series of weight $2$ for the congruence subgroup $\Gamma_0(p)\subseteq \PGL_2(\Z)$.
The $p$-adic Galois representation $\rho_{E_2(p)}$ attached to $E_2(p)$ is the direct sum $\Q_p\oplus \Q_p(1)$ and, thus, we have
$$\Lambda_{\rho_{E_2(p),p}}=\left\{0\right\}\subseteq \Q_p^\times.$$
Therefore we expect that there should be a class in $\HH^1(\Gamma^p,\mathcal{A}^\times)$ attached to $E_2(p)$.
But it turns out that the non-existence of the Eisenstein series of weight $2$ and level $1$ implies that no such class exists.
We will sketch in the following how our results together with the strategy of \cite{GeR} gives a class
$$J_{DR}\in \HH^1(\Gamma^p,\mathcal{A}^\times/ p^\Z)\otimes \Q.$$
This class was recently constructed by Darmon--Pozzi--Vonk using Siegel units (see \cite{DPV2}, Theorem A).

Clearly, we have
\begin{align*}
\HH^{0}(\PGL_2(\Z),\Z)=\Z,\ \HH^0(\Gamma_0(p),\Z)^{W_p=-1}=0
\ \mbox{and}\ 
\HH^1(\PGL_2(\Z),\Z)=0.
\end{align*}
Therefore, the long exact sequence associated with \eqref{Steinbergresolution} induces the short exact sequence
$$0\too \HH^{0}(\PGL_2(\Z),\Z)\too \HH^{1}(\Gamma^p, \St_p(\Z)^\ast)\too \HH^1(\Gamma_0(p),\Z)^{W_p=-1}\too 0.$$
We define $J_{triv}\in \HH^{1}(\Gamma^p, \St_p(\Z)^\ast,\Z)$ to be the image of the canonical generator under the left arrow. 

Furthermore, we see that
$$\HH^{i}(\Gamma^p, \St_p(\Z)^\ast)=0\quad \forall i\neq 1.$$
Next, by analyzing the long exact sequence coming from \eqref{Eisensteinpart} we deduce that
$$\HH^{0}(\Gamma^p, \mathcal{E}(\ord_p)^\ast)=\HH^{1}(\Gamma^p, \mathcal{E}(\ord_p)^\ast)=\Z$$
and
$$\HH^{i}(\Gamma^p, \mathcal{E}(\ord_p)^\ast)=0\quad \forall i\neq 0,1.$$

\begin{Rem}
Note that only the cohomology of arithmetic groups in degree $0$ - or in other words the trivial representation - contributes to the cohomology of $\mathcal{E}(\ord_p)^\ast$.
If the Eisenstein series of weight $2$ and level $1$ would exist, then it would also contribute to the cohomology of $ \mathcal{E}(\ord_p)^\ast$.

One can enlarge $\HH^{1}(\Gamma^p, \mathcal{E}(\ord_p)^\ast)$ by raising the level away from $p$.
This process is often called smoothing (see for example the constructions of Darmon--Dasgupta in \cite{DD}).
\end{Rem}

It is well-known that
$$\HH^{1}(\Gamma^p,\Z)=0.$$
Thus the long exact sequence associated with \eqref{ordext} induces the short exact sequence
$$0\too \HH^{1}(\Gamma^p, \mathcal{E}(\ord_p)^\ast)\too \HH^{1}(\Gamma^p, \St_p(\Z)^\ast)\too \HH^2(\Gamma^p,\Z)\too 0.$$
By the discussion in Section 3.4 of \cite{Sp} the exact sequences \eqref{Steinbergresolution}, \eqref{ordext} and \eqref{Eisensteinpart} all fit together nicely into the following big commutative diagram with exact rows and columns:
\begin{center}
 \begin{tikzpicture}
    \path 	
		(0,0) node[name=A]{$0$}
		(2,0) node[name=B]{$0$}
		(-3.7,-1) node[name=C]{$0$}
		(-2.2,-1) node[name=D]{$\cind_{K_p}^{G_p}\Z$}
		(0,-1) node[name=E]{$\cind_{\widetilde{I}_p}^{G_p}\chi_p$}
		(2,-1) node[name=F]{$\St_p(\Z)$}
		(3.5,-1) node[name=G]{$0$}
		(-3.7,-2.2)node[name=H]{$0$}
		(-2.2,-2.2) node[name=I]{$\cind_{K_p}^{G_p}\Z$}
		(0,-2.2) node[name=J]{$\cind_{K_p}^{G_p}\Z$}
		(2,-2.2) node[name=K]{$\mathcal{E}(\ord_p)$}
		(3.5,-2.2) node[name=L]{$0$}
		(0,-3.4) node[name=M]{$\Z$}
		(2,-3.4) node[name=N]{$\Z$}
		(0,-4.3) node[name=O]{$0$}
		(2,-4.3) node[name=P]{$0$}
	   ;
		
		\draw[->] (A) -- (E) ;
    \draw[->] (B) -- (F) ;
		
		\draw[->] (C) -- (D) ;
		\draw[->] (D) -- (E) ;
		\draw[->] (E) -- (F) ;
    \draw[->] (F) -- (G) ;
		
		\draw[->] (D) -- (I) node[midway,right]{$=$};
		\draw[->] (E) -- (J) ;
    \draw[->] (F) -- (K) ;
		
		\draw[->] (H) -- (I) ;
		\draw[->] (I) -- (J) ;
		\draw[->] (J) -- (K) ;
    \draw[->] (K) -- (L) ;
		
		\draw[->] (M) -- (N) node[midway,above]{$=$};
		\draw[->] (J) -- (M) ;
    \draw[->] (K) -- (N) ;
		
		\draw[->] (M) -- (O) ;
    \draw[->] (N) -- (P) ;
  \end{tikzpicture} 
\end{center}

This immediately implies the image of the generator of $\HH^{1}(\Gamma^p, \mathcal{E}(\ord_p)^\ast)$ under the map
$$\HH^{1}(\Gamma^p, \mathcal{E}(\ord_p)^\ast)\too \HH^{1}(\Gamma^p, \St_p(\Z)^\ast)$$
is equal to $J_{triv}$.
Thus, we get canonical isomorphisms
$$\HH^{1}(\Gamma_0(p),\Z)\xlongleftarrow{\cong}\HH^{1}(\Gamma^p, \St_p(\Z)^\ast)/<J_{triv}>\xlongrightarrow{\cong} \HH^2(\Gamma^{p},\Z).$$

Remember that van der Put's theorem gives a canonical isomorphism
$$\HH^{1}(\Gamma^p, \St_p(\Z)^\ast)=\HH^{1}(\Gamma^p, \mathcal{A}^\times/\Q_p^\times).$$
We denote the image of $J_{triv}$ under this isomorphism also by $J_{triv}.$
The Eisenstein series $E_2(p)$ defines a class in $\HH^{1}(\Gamma_0(p),\Z)$ and, thus, also a class
$$J_{DR}\in \HH^{1}(\Gamma^p, \mathcal{A}^\times/\Q_p^\times)$$
that is unique up to powers of $J_{triv}.$
Arguing as in the previous section we see that this class cannot be lifted to a class in $\HH^{1}(\Gamma^p, \mathcal{A}^\times)$.
But there exists a discrete subgroup
$$\Lambda_{E_2(p),p}\subseteq \Q_p^\times$$
of rank one such that $J_{DR}$ can be lifted to a class in $\HH^{1}(\Gamma^p, \mathcal{A}^\times/\Lambda_{E_2(p),p}).$
Thus, we have to show that this discrete subgroup is homothetic to the one generated by $p$.

Although Theorem 3.16 of \cite{GeR} is stated only for cusp forms it is enough to assume that the system of eigenvalues for the full Hecke algebra (including the Hecke operator at $p$) shows up only in a single degree in cohomology.
This is certainly the case for $E_2(p).$
Thus one can argue as in \textit{loc.cit.~}to show that $\Lambda_{E_2(p),p}$ is homothetic to the lattice generated by any element $q\in \Q_\p^\times$ such that
\begin{itemize}
\item $\ord_p(q)\neq 1$ and 
\item $\log_p(q)=-2 \left.\frac{d\alpha}{dk}\hspace{-0.2em}\mid_{k=2}\right.$
\end{itemize}
where $\alpha$ is the the $U_p$-eigenvalue of the Hida family passing through $E_2(p)$.
As the $U_p$-eigenvalue in the Eisenstein family is constant (and equal to $1$) we see that $q=p$ fulfils the two properties above.

\subsection{Theta cocycles for definite unitary groups}
We end this note with some remarks about theta cocycles for definite unitary groups.
The author hopes to return to the study of these cocycles and their arithmetic applications in the future.
Let $GU$ be the group of unitary similitudes of a definite hermitian form over a totally real number field $F$ and let $G$ be $GU$ modulo its centre.
Assume that $G$ is split at a prime $\p$ of $F$, i.e. $G(F_\p)\cong \PGL_n(F_\p).$

Let $\pi$ be an automorphic representation which is cohomological with respect to the trivial coefficient system and such that its local component at $\pi_\p$ is the Steinberg representation of $\PGL_n(F_\p)$.
We assume for simplicity that all Hecke eigenvalues are rational.
By the result of many authors (see Theorem 2.1.1 of \cite{BLGGT} for an overview), one can attach a Galois representation $\rho_\pi$ to $\pi$ and the restriction $\rho_{\pi,\p}$ of $\rho_\pi$ to a decomposition group at $\p$ is of the following form: it is upper triangular and the $i$-th diagonal entry is the $n-i$-th cyclotomic character.

Let $\rho_{\pi,\p,1}$ be the $2$-dimensional subrepresentation given by the upper-left-$2\times2$-block of that matrix, i.e., $\rho_{\pi,\p,1}$ is an extension of $\Q_p(n-1)$ by $\Q_p(n-2)$.
It is known that this extension is not crystalline.
Thus, we can define the lattice
$$\Lambda_{\rho}=\Lambda_{\rho_{\pi,\p,1}(n)}.$$

Let $\mathcal{A}$ be the ring of analytic function on Drinfeld's upper half space of dimension $n-1$ over $F_\p$.
Under suitable strong multiplicity one assumptions \cite{Ge4}, Proposition 3.9, together with Gekeler's isomorphism implies that, for appropriate $\p$-arithmetic subgroups $\Gamma^{\p}\subseteq G(F)$, the $\pi$-isotypic part of $\HH^{\ast}(\Gamma^{\p},\mathcal{A}^{\times}/F_\p^{{\times}})\otimes \Q$ is concentrated in degree $n-2$ and non-zero.
Using \cite{GeR}, Theorem 4.3, and  \cite{Ge4}, Theorem 3.19, we see that (under mild assumptions on $\rho$) the map
$$(\HH^{n-2}(\Gamma^{\p},\mathcal{A}^{\times}/\Lambda_{\rho})\otimes \Q)^{\pi}\too (\HH^{n-2}(\Gamma^{\p},\mathcal{A}^{\times}/F_\p^{{\times}})\otimes \Q)^{\pi}$$
is an isomorphism and that
$$(\HH^{n-2}(\Gamma^{\p},\mathcal{A}^{\times})\otimes \Q)^{\pi}=0.$$

\def\cprime{$'$}

\end{document}